\documentclass[12pt,a4paper]{amsart}[2000]
\usepackage{geometry}
\usepackage{tikz-cd}
\usepackage{amsmath}
\usepackage{amssymb}
\usepackage{amsthm}
\usepackage{mathrsfs}
\usepackage[hang]{footmisc}
\usepackage{hyperref}
\usepackage{color}
\usepackage{mathtools}
\usepackage{enumitem}
\usepackage{stackengine}

\title[non-free almost finite actions]{Non-free almost finite actions for locally finite-by-virtually $\Z$ groups}

\thanks{}

\allowdisplaybreaks

\theoremstyle{plain}

\newtheorem{Thm}{Theorem}[section]

\theoremstyle{definition}

\theoremstyle{plain}
\newtheorem{thm}[Thm]{Theorem}
\newtheorem{lem}[Thm]{Lemma}
\newtheorem{cor}[Thm]{Corollary}
\newtheorem{prop}[Thm]{Proposition}
\newtheorem{conj}[Thm]{Conjecture}

\theoremstyle{definition}
\newtheorem{defn}[Thm]{Definition}
\newtheorem{ques}[Thm]{Question}

\newtheorem{rmk}[Thm]{Remark}

\newtheorem{innercustomthm}{Theorem}
\newenvironment{customthm}[1]
{\renewcommand\theinnercustomthm{#1}\innercustomthm}
{\endinnercustomthm}

\newcommand{\A}[0]{\mathbb{A}}
\newcommand{\B}{B}
\newcommand{\J}{J}
\newcommand{\K}{\mathcal{K}}
\newcommand{\D}{D}
\newcommand{\Ch}{D}
\newcommand{\Zh}{\mathcal{Z}}
\newcommand{\E}{E}
\newcommand{\Oh}{\mathcal{O}}

\newcommand{\T}{{\mathbb T}}

\newcommand{\N}{{\mathbb N}}
\newcommand{\Z}{{\mathbb Z}}
\newcommand{\C}{{\mathbb C}}
\newcommand{\Q}{{\mathbb Q}}

\newcommand{\aut}{\mathrm{Aut}}
\newcommand{\supp}{\mathrm{supp}}

\newcommand{\eps}{\varepsilon}
\numberwithin{equation}{section}

\newcommand{\id}{\mathrm{id}}

\newcommand{\halpha}{\widehat{\alpha}}
\newcommand{\calpha}{\widehat{\alpha}}

\newcommand{\tih}{\widetilde {h}}

\newcommand\set[1]{\left\{#1\right\}}  
\newcommand\mset[1]{\left\{\!\!\left\{#1\right\}\!\!\right\}}

\newcommand{\CA}[0]{\mathcal{A}} \newcommand{\CB}[0]{\mathcal{B}}
\newcommand{\CC}[0]{\mathcal{C}} \newcommand{\CD}[0]{\mathcal{D}}
 
\newcommand{\CG}[0]{\mathcal{G}} \newcommand{\CH}[0]{\mathcal{H}}

\newcommand{\CQ}[0]{\mathcal{Q}} 
 \newcommand{\CT}[0]{\mathcal{T}}
 
\newcommand{\CW}[0]{\mathcal{W}} 
 \newcommand{\CZ}[0]{\mathcal{Z}}

\newcommand{\Ra}[0]{\Rightarrow}
\newcommand{\La}[0]{\Leftarrow}
\newcommand{\LRa}[0]{\Leftrightarrow}

\newcommand{\quer}[0]{\overline}
\newcommand{\eins}[0]{\mathbf{1}}			
\newcommand{\diag}[0]{\operatorname{diag}}
\newcommand{\ad}[0]{\operatorname{Ad}}
\newcommand{\ev}[0]{\operatorname{ev}}
\newcommand{\fin}[0]{{\subset\!\!\!\subset}}
\newcommand{\diam}[0]{\operatorname{diam}}
\newcommand{\Hom}[0]{\operatorname{Hom}}
\newcommand{\dst}[0]{\displaystyle}
\newcommand{\spp}[0]{\operatorname{supp}}
\newcommand{\lsc}[0]{\operatorname{Lsc}}
\newcommand{\del}[0]{\partial}
\newcommand{\GU}[0]{\CG^{(0)}}

\theoremstyle{definition}

\numberwithin{equation}{Thm}

\begin{document}
	\global\long\def\floorstar#1{\lfloor#1\rfloor}
	\global\long\def\ceilstar#1{\lceil#1\rceil}	
	
	\global\long\def\B{B}
	\global\long\def\A{A}
	\global\long\def\J{J}
	\global\long\def\K{\mathcal{K}}
	\global\long\def\D{D}
	\global\long\def\Ch{D}
	\global\long\def\Zh{\mathcal{Z}}
	\global\long\def\E{E}
	\global\long\def\Oh{\mathcal{O}}

	\global\long\def\T{{\mathbb{T}}}
	\global\long\def\BR{{\mathbb{R}}}
	\global\long\def\N{{\mathbb{N}}}
	\global\long\def\Z{{\mathbb{Z}}}
	\global\long\def\C{{\mathbb{C}}}
	\global\long\def\Q{{\mathbb{Q}}}

	\global\long\def\aut{\mathrm{Aut}}
	\global\long\def\supp{\mathrm{supp}}

	\global\long\def\eps{\varepsilon}

	\global\long\def\id{\mathrm{id}}

	\global\long\def\halpha{\widehat{\alpha}}
	\global\long\def\calpha{\widehat{\alpha}}

	\global\long\def\tih{\widetilde{h}}

	\global\long\def\opFol{\operatorname{F{\o}l}}

	\global\long\def\opRange{\operatorname{Range}}

	\global\long\def\opIso{\operatorname{Iso}}

	\global\long\def\dimnuc{\dim_{\operatorname{nuc}}}

	\global\long\def\set#1{\left\{  #1\right\}  }

	
	\global\long\def\mset#1{\left\{  \!\!\left\{  #1\right\}  \!\!\right\}  }

	\global\long\def\Ra{\Rightarrow}
	\global\long\def\La{\Leftarrow}
	\global\long\def\LRa{\Leftrightarrow}

	\global\long\def\quer{\overline{}}
	\global\long\def\eins{\mathbf{1}}
	\global\long\def\diag{\operatorname{diag}}
	\global\long\def\ad{\operatorname{Ad}}
	\global\long\def\ev{\operatorname{ev}}
	\global\long\def\fin{{\subset\!\!\!\subset}}
	\global\long\def\diam{\operatorname{diam}}
	\global\long\def\Hom{\operatorname{Hom}}
	\global\long\def\dst{{\displaystyle }}
	\global\long\def\spp{\operatorname{supp}}
	\global\long\def\spo{\operatorname{supp}_{o}}
	\global\long\def\del{\partial}
	\global\long\def\lsc{\operatorname{Lsc}}
	\global\long\def\GU{\CG^{(0)}}
	\global\long\def\HU{\CH^{(0)}}
	\global\long\def\AU{\CA^{(0)}}
	\global\long\def\BU{\CB^{(0)}}
	\global\long\def\CUU{\CC^{(0)}}
	\global\long\def\DU{\CD^{(0)}}
	\global\long\def\QU{\CQ^{(0)}}
	\global\long\def\TU{\CT^{(0)}}
	\global\long\def\CUUU{\CC'{}^{(0)}}
	
	\global\long\def\bA{\mathbb{A}}
	\global\long\def\AUl{(\CA^{l})^{(0)}}
	\global\long\def\BUl{(B^{l})^{(0)}}
	\global\long\def\HUp{(\CH^{p})^{(0)}}
	\global\long\def\sym{\operatorname{Sym}}
	
	\global\long\def\properlength{proper}

	\global\long\def\interior#1{#1^{\operatorname{o}}}

\author{Kang Li}

\address{K. Li: Department of Mathematics, Friedrich-Alexander-Universität Erlangen-Nürnberg,  Cauerstraße 11, 91058 Erlangen, Germany}
\email{kang.li@fau.de}

	\author{Xin Ma}
	
	\address{X. Ma: Department of Mathematics and Statistics, York University, Toronto, ON, Canada, M3J 1P3}
  \email{xma17@yorku.ca}	
	

	
	\date{\today}

	\begin{abstract}
		In this paper, we study almost finiteness and almost finiteness in measure of non-free actions.  Let $\alpha:G\curvearrowright X$ be a minimal action of a locally finite-by-virtually $\Z$ group $G$ on the Cantor set $X$. We prove that under certain assumptions, the action $\alpha$  is almost finite in measure if and only if $\alpha$ is essentially free. As an application, we obtain that any minimal topologically free action of a virtually $\Z$ group on an infinite compact metrizable space with the small boundary property is almost finite. This is the first general result, assuming only topological freeness, in this direction, and these lead to new results on uniform property $\Gamma$ and $\CZ$-stability for their crossed product $C^*$-algebras. Some concrete examples of minimal topological free (but non-free) subshifts are provided.

	\end{abstract}
	\maketitle

 \section{Introduction}
In recent years, there has been an increasing acknowledgment of the profound interplay between the field of topological dynamical systems and $C^*$-algebras. Topological dynamical systems have emerged as a valuable source of examples and motivations for exploring $C^*$-algebras, particularly through the construction of crossed product $C^*$-algebras. In particular, many dynamical properties such as profinite-ness, mean dimension zero, small boundary property, almost finiteness (in measure), pure infiniteness, and certain geometric properties of acting groups like sub-exponential growth, and elementary amenability, have found applications in determining the useful structural property of the crossed products, e.g., $\CZ$-stability, which plays a central role in the study of the classification of nuclear simple separable $C^*$-algebras. We refer to, e.g., \cite{E-N}, \cite{D}, \cite{K-N}, \cite{K-S}, \cite{M1}, \cite{M-Wang}, and \cite{Niu}, for several recent developments in this direction.

Almost finiteness (Definition \ref{defn: af}) was first introduced in \cite{Matui} by Matui in the setting of ample \'{e}tale groupoids to study the homology theory of groupoids and their topological full groups and then refined by Kerr in the framework of dynamical systems $\alpha: G\curvearrowright X$ in \cite{D}, where $G$ is a countable amenable group and $X$ is an infinite compact metrizable space. This property, together with its weaker form, almost finiteness in measure introduced in \cite{K-S} (Definition \ref{defn: af in measure}), can be regarded as topological dynamical analogues of the well-known Ornstein-Weiss quasi-tilling theorem in ergodic theory (see, e.g., \cite[Theorem 4.46]{Kerr-L}). In \cite{D}, Kerr proved that almost finiteness of a minimal free action $\alpha: G\curvearrowright X$ implies that $C(X)\rtimes_r G$ is $\CZ$-stable (see Theorem \ref{thm: af Z stability}). In addition, it was demonstrated in \cite{K-S} that almost finiteness in measure, on the other hand, entails that the crossed product $C(X)\rtimes_r G$ possesses uniform property $\Gamma$ and thus satisfies the Toms-Winter conjecture via the results in \cite{C-E-T-W}.
Motivated by these, for an action $\alpha: G\curvearrowright X$  of an amenable group $G$ on an infinite compact metrizable space $X$,
it is natural to ask when the action $\alpha$ is almost finite or almost finite in measure.

On the other hand, the classification program of $C^*$-algebras (see Theorem \ref{thm: classification} below) deals with nuclear simple separable $C^*$-algebras. For a crossed product $C(X)\rtimes_r G$, it was demonstrated in \cite{AS} that $C(X)\rtimes_r G$ is nuclear and simple if and only if the action is minimal, topologically free and amenable. Note that all actions of an amenable group are amenable. Therefore, it seems that minimal topologically free actions from amenable groups are candidates with full potential to yield nuclear stably finite crossed products that fit the classification theorem.  In this context, the remaining question is verifying the almost finiteness of these actions. We remark that the small boundary property (Definition \ref{defn: SBP}), as a dynamical analogue for a zero-dimensional space, is a necessary condition of almost finiteness by \cite[Theorem 5.5]{K-S}. See also \cite[Proposition 3.8]{M1}. Therefore, the question above boils down to the following question.

\begin{ques}\label{ques2}
   Let $\alpha: G\curvearrowright X$ be a minimal topologically free action of an amenable group $G$ on an infinite compact metrizable space $X$ with the small boundary property. When is the action $\alpha$ almost finite or almost finite in measure?
\end{ques}

So far, there have been several results addressing Question \ref{ques2} under the assumption that the action is free. See \cite{D}, \cite{K-S}, \cite{K-N} and \cite{Niu}. In particular, it was demonstrated in \cite{K-N} that all minimal free actions of elementary amenable groups on finite-dimensional spaces are almost finite. However, it was shown in \cite{Joseph} that it is not the case if one looks at topologically free actions. To be more specific, it was shown in \cite{Joseph} that certain wreath products, like $\Z^d\wr \Z$, admit profinite topologically free but not essentially free (see Section \ref{sec2} for the definition) actions.  On the other hand, essential freeness is a necessary condition for almost finiteness, first observed in the groupoid setting in \cite[Remark 2.4]{O-S} whose authors attribute to Matui \cite[Remark 6.6]{Matui}. See also \cite[Lemma 2.2]{Joseph}. Therefore, such actions cannot be almost finite. See \cite{H-W} for more such exotic examples of actions with the same flavor, such as certain profinite actions of $G\wr H$, where $G$ is a non-trivial countable abelian residually finite group and $H$ is a countably infinite residually finite group.

A minimal topologically free action is called  \textit{allosteric} if it is not essentially free. A group $\Gamma$ is said to be allosteric if it admits allosteric actions.  So far several allosteric and non-allosteric examples have been found (see Remark \ref{rmk: allosteric} below). 

Therefore, motivated by the results above,
it is worth working on Question \ref{ques2} when the action $\alpha: G \curvearrowright X$ is essentially free even in the case that the acting group $G$ is elementary amenable and the case that the acting group is not allosteric. In this direction, for odometers of an amenable group $G$, it is demonstrated in \cite[Theorem 2.5]{O-S} that $\alpha: G\curvearrowright X$ is almost finite if and only if $\alpha$ is essentially free. It is also proved in \cite[Theorem 2.10]{O-S} that all minimal actions of the infinite dihedral group $D_\infty=\Z_2*\Z_2$ on the Cantor set are almost finite. However, it seems that there are no results on spaces with the infinite covering dimension, which is addressed in this paper. 

In general, without the genuine freeness of the action, the main difficulty in establishing almost finiteness and almost finiteness in measure is that it is not clear how to build disjoint towers with F{\o}lner shapes properly. To the best knowledge of authors, the only known way in \cite{O-S} establishing almost finiteness in the non-free setting depends on the specific structure of profinite actions. However,
we somehow overcome this by using certain permanence properties with respect to certain group extensions, which is the main novelty of the current paper. We first show in Proposition \ref{prop: af in measure imply ess free} that essentially freeness is even a necessary condition of almost finiteness in measure in general, which generalizes \cite[Remark 2.4]{O-S} and \cite[Lemma 2.2]{Joseph}. If we look at the case that the acting group is locally finite-by-virtually $\Z$, we have the following result on the equivalence of almost finiteness in measure with essentially freeness for certain actions.

\begin{customthm}{A}(Corollary \ref{cor: af in measure more general})\label{thm: main 1}
Let \[\begin{tikzcd}
0\arrow[r] & F \arrow[r, "i"] & G\arrow[r, "\rho"] & H\arrow[r] & 0
\end{tikzcd}\]
be an extension of countable discrete groups $F, H$, and $G$, in which $F$ is locally finite and $G$ is virtually $\Z$. Suppose $\alpha: G \curvearrowright X$ is a minimal action on 
the Cantor set $X$ such that $X/F$ is Hausdorff.  Then $\alpha$ is almost finite in measure if and only if it is essentially free.
\end{customthm}

We remark that the class of acting groups considered in Theorem \ref{thm: main 1} has some overlap, e.g. lamplighter group $\Z_2\wr \Z$, with the class of groups in \cite[Section 11]{H-W}, which allows profinite allosteric actions. This is one of our motivations for considering this class of groups. Note that \cite[Proposition 11.4, Remark 11.5]{H-W} showed that the crossed product $C(X)\rtimes_r (H\wr \Z)$ of profinite actions on the Cantor set $X$ has nuclear dimension one, in which $H$ is abelian, locally finite and residually finite. Thus, it is interesting to compare this with our Theorem \ref{thm: main 1} because the minimal essentially free action in Theorem \ref{thm: main 1} on the Cantor set of such a group, on the other hand,  yields a crossed product satisfying uniform property $\Gamma$ via Theorem \ref{thm: af in measure Gamma} (see Theorem \ref{thm: main3} below). 

However, if we consider groups with ``lower complexity'', i.e., virtually $\Z$ groups, then 
we may show almost finiteness instead of the ``in measure'' version even for actions on compact metrizable spaces with the small boundary property. We establish this through the fact that almost finiteness is equivalent to almost finiteness in measure together with dynamical comparison. This equivalence has been proved under assumptions of minimality or freeness in \cite{D} and \cite{K-S}. Actually, this equivalence holds in general, which might be already known to experts. However, it seems that it has not appeared anywhere in the literature. Therefore, for the reader's convenience, we include the proof in Propositions \ref{rmk: af in measure plus comparison equal af} and \ref{prop: af imply af in measure and comparison}. 

Moreover, since virtually $\Z$ groups are non-allosteric, we obtain almost finiteness for minimal topologically free actions on the space with the small boundary property. This answers Question \ref{ques2} in full strength for virtually $\Z$ groups, which is the main result in the current paper. We remark that to the best knowledge of the authors, our Theorem \ref{thm: main2} is the first general result, assuming only topological freeness, in this direction.

\begin{customthm}{B}(Corollary \ref{cor: virtually cyclic af})\label{thm: main2}
     Let $\alpha: G\curvearrowright X$ be a minimal topologically free action of a virtually $\Z$ group $G$ on an infinite compact metrizable space $X$ with the small boundary property.  Then $\alpha$ is almost finite.\end{customthm}

Note that our Theorem \ref{thm: main2} (together with Theorem \ref{thm: main3}(ii)) has extended \cite[the Theorem]{E-N} by Elliott-Niu from minimal (free) $\Z$-actions to virtually $\Z$ minimal topologically free actions.
On the other hand, our Theorem \ref{thm: main2} has also generalized the case of minimal $D_\infty$-action on the Cantor set in \cite[Theorem 2.10]{O-S} because all minimal $D_\infty$-actions are topologically free as shown in \cite[Proposition 2.6]{J} (see also \cite[Proposition 2.8]{O-S}). One may want to compare our Theorem \ref{thm: main2} with another approach to study the actions of virtually $\Z$ groups based on the dynamical asymptotic dimension by the first author in \cite{ALSS} and \cite{BL}. The following are the standard applications of Theorems \ref{thm: main 1} and \ref{thm: main2} to $C^*$-algebras.

\begin{customthm}{C}(Corollary \ref{cor: final})\label{thm: main3}
 \begin{enumerate}[label=(\roman*)]
    \item Let \[\begin{tikzcd}
0\arrow[r] & F \arrow[r, "i"] & G\arrow[r, "\rho"] & H\arrow[r] & 0
\end{tikzcd}\]
be an extension of countable discrete groups $F, H$, and $G$, in which $F$ is locally finite and $H$ is virtually $\Z$. Suppose $G \curvearrowright X$ is a minimal action on the Cantor set $X$ such that the quotient space $X/F$ is Hausdorff. Then $C(X)\rtimes_r G$ has uniform property $\Gamma$ and thus satisfies the Toms-Winter conjecture. 
    \item Suppose $\alpha: G\curvearrowright X$ is a minimal
    topologically free action of a 
    virtually $\Z$ group $G$ on an infinite compact metrizable space $X$ with the small boundary property, then $C(X)\rtimes_r G$ is $\CZ$-stable and thus classifiable by its Elliott invariant.    \end{enumerate}
\end{customthm}

In Section \ref{sec4}, we construct concrete examples of minimal topological free (but non-free) subshifts of $D_\infty$ and $D_\infty\times F$, where $F$ is a finite group, such that our Theorems \ref{thm: main2} and \ref{thm: main3} can be applied. These provide additional examples than odometers considered in \cite{O-S}. We end the introduction with the following conjecture motivated by \cite{K-N}, Theorem \ref{thm: main2} and Theorem \ref{thm: main3}.

\begin{conj}
    All minimal essentially free actions of elementary amenable groups on infinite compact metrizable spaces with the small boundary property are almost finite, whose crossed products are thus $\CZ$-stable.
\end{conj}

 \section{Preliminaries}\label{sec2}
 In this section, we recall backgrounds and some standard facts. Throughout the paper, we only consider countable amenable discrete group $\Gamma$ and infinite compact Hausdorff space $X$, which is usually assumed to be metrizable. In addition, all actions $\alpha: \Gamma\curvearrowright X$ are assumed to be actions by homeomorphisms. 
 However, in the context of crossed product $C^*$-algebras, especially when uniform property $\Gamma$ is involved, we denote $G$ for the acting group. 
 
 Let $\alpha: \Gamma\curvearrowright X$ be an action. We denote by $M_\Gamma(X)$ for the set of all $\Gamma$-invariant Borel probability measures on $X$.  Let  $K$ be a finite set in $\Gamma$ and $\epsilon>0$. A finite set $S\subset \Gamma$ is said to be $(K, \epsilon)$-\textit{F{\o}lner} if $|KS\setminus S|<\epsilon|S|$. An action $\alpha: \Gamma\curvearrowright X$ is said to be \textit{minimal} if all orbits are dense in $X$. Such an $\alpha$ is said to be \textit{free} if the stabilizer group $\operatorname{stab}(x)$ for each $x\in X$ is trivial.  An action $\alpha$ is said to be \textit{topologically free} (resp. \textit{essentially free}) if for any non-trivial $s\neq e$, the fixed point set $F_s=\{x\in X: sx=x\}$ is nowhere dense (resp. satisfies $\sup_{\mu\in M_\Gamma(X)}\mu(F_s)=0$). For minimal actions, it is straightforward to see that an essentially free action is topologically free. On the other hand, a minimal action is said to be \textit{allosteric} if it is topologically free but not essentially free. A countable group $\Gamma$ is called \textit{allosteric} if it admits an allosteric action. 
 
 \begin{rmk}\label{rmk: allosteric}
     So far, many non-allosteric groups have been found in the literature.
     For example, it was proved in \cite[Corollary 2.4]{Jo2} that if a group $\Gamma$ has only countably many subgroups, then $\Gamma$ is non-allosteric. 
     This, in particular, applies to virtually $\Z$ groups. See  \cite{Joseph} and \cite{H-W} for more information on allosteric and non-allosteric groups.  
         
     \end{rmk}

We now recall a comparison property in dynamical systems. Throughout the paper, we write ``$\bigsqcup$'' and ``$\sqcup$'' for disjoint unions.

\begin{defn}\cite[Definition 3.1, 3.2]{D}\label{defn: comparison}
    Let $\alpha: \Gamma\curvearrowright X$ be an action. Let $K$ be a closed set, $O, U$ non-empty open sets in $X$.
    \begin{enumerate}[label=(\roman*)]
        \item We write $K\prec O$ if there exists open sets $V_1, \dots, V_n$ in $X$ and $s_1,\dots, s_n\in \Gamma$ such that $K\subset \bigcup_{i=1}^nV_i$ and $\bigsqcup_{i=1}^ns_iV_i\subset O$.
        \item We denote by $U\prec O$ if $K\prec O$ holds for any closed $K\subset U$.
        \item We say $\alpha$ has \textit{dynamical comparison} if $U\prec O$ whenever $\mu(U)<\mu(O)$ holds for any $\mu\in M_\Gamma(X)$.
    \end{enumerate}
\end{defn}
 
 The following notions are fundamental in the study of topological dynamical systems and ergodic theory.

\begin{defn}
Let $\alpha: \Gamma\curvearrowright X$ be an action and $S\subset \Gamma$ finite and $B$ a set in $X$. A pair $\CT=(S, B)$ is called a \textit{tower} if  $\{sB: s\in S\}$ is a disjoint family in which $S$ is called the \textit{shape} and $B$ is called the \textit{base} of $\CT$. We say the tower $\CT$ is open (resp. closed) if the base $B$ is open (resp. closed).
\end{defn}

\begin{defn}
Let $\alpha: \Gamma\curvearrowright X$ be an action and $\CC=\{\CT_i=(S_i, B_i): i\in I\}$ an finite family of towers. We say $\CC$ is a \textit{castle} if each pair of $\CT_i$ and $\CT_j$ are disjoint in the sense that $sB_i\cap tB_j=\emptyset$ for any $s\in S_i, t\in T_j$ whenever $i\neq j\in I$. We also say $\CC$ is open (resp. closed) if all towers $\CT_i$ in $\CC$ are open (resp. closed).
\end{defn}

 We now recall the definition of almost finiteness and almost finiteness in measure.
 
 \begin{defn}\label{defn: af}\cite[Definition 8.2]{D}
     Let $\alpha: \Gamma\curvearrowright X$ be an action of an amenable group $\Gamma$. The action $\alpha$ is said to be almost finite if, for any finite set $K\subset \Gamma$, integer $n\in \N$ and $\epsilon>0$, there exists an open castle $\CC=\{(S_i, B_i): i\in I\}$ such that 
     \begin{enumerate}[label=(\roman*)]
         \item each shape $S_i$ is $(K, \epsilon)$-F{\o}lner;
         \item the diameter $\diam(sB_i)<\epsilon$ for any $s\in S_i$ and $i\in I$;
         \item for each $i\in I$ there exists $S'_i\subset S_i$ with $|S'_i|<(1/n)|S_i|$ such that one has $X\setminus \bigsqcup_{i\in I}S_iB_i\prec \bigsqcup_{i\in I}S'_iB_i$
     \end{enumerate}
 \end{defn}

 \begin{rmk}\label{rmk: zero dimensional af}
     We remark that when $X$ is zero-dimensional, with the help of \cite[Theorem 10.2]{D}, one refines Definition \ref{defn: af} above in the following sense.
     \begin{enumerate}[label=(\roman*)]
         \item The castle $\CC$ is clopen.
         \item  The remainder of the castle is empty, i.e.,  $X\setminus \bigsqcup_{i\in I}S_iB_i=\emptyset$.     \end{enumerate}
 \end{rmk}

\begin{defn}\cite[Definition 3.5]{K-S}\label{defn: af in measure}
 Let $\alpha: \Gamma\curvearrowright X$ be an action of an amenable group $\Gamma$. The action $\alpha$ is said to be almost finite in measure if, for any finite set $K\subset \Gamma$, and $\epsilon>0$, there exists an open castle $\CC=\{(S_i, B_i): i\in I\}$ such that 
     \begin{enumerate}[label=(\roman*)]
         \item each shape $S_i$ is $(K, \epsilon)$-F{\o}lner;
         \item the diameter $\diam(sB_i)<\epsilon$ for any $s\in S_i$ and $i\in I$;
         \item we have $\sup_{\mu\in M_\Gamma(X)}\mu(X\setminus \bigsqcup_{i\in I}S_iB_i)<\epsilon$.
     \end{enumerate}
     Similarly to Remark \ref{rmk: zero dimensional af}, in the zero-dimensional setting, one may ask the castle $\CC$ to be clopen.
\end{defn}

It is not hard to see that if an (not necessarily free) action $\alpha$ is almost finite in measure and has dynamical comparison, then $\alpha$ is almost finite. For the reader's convenience, we include the proof here.


\begin{prop}\label{rmk: af in measure plus comparison equal af}
   Let $\alpha: \Gamma\curvearrowright X$ be an action of an infinite amenable group $\Gamma$ on a compact metrizable space $X$. Suppose $\alpha$ is almost finite in measure and has dynamical comparison. Then $\alpha$ is almost finite.
\end{prop}
\begin{proof}
    Let finite $K\subset \Gamma$, $\epsilon>0$ and $n\in \N_+$ be given. Choose another finite $K'\subset \Gamma$ with $K\subset K'$ and $|K'|> 2n(n+1)$. Write $\delta=\min\{\epsilon, \frac{1}{n+2}\}$ for simplicity.
    Note that our choice of $K'$ and $\delta$ implies that any $(K', \delta)$-F{\o}lner set $S$ is also $(K, \epsilon)$-F{\o}lner.    
    Moreover, any such $S$ satisfies that $|K'S|\leq (1+\delta)|S|$ and thus one has
    \[|S|\geq \frac{|K'S|}{1+\delta}\geq \frac{|K'|}{2}> n(n+1),\]
    which implies that 
    \[\frac{|S|}{n}-\frac{|S|}{n+1}> 1.\]
    Therefore, there is an integer $m_S\in (\frac{|S|}{n+1}, \frac{|S|}{n})$ for any $(K', \delta)$-F{\o}lner set $S$.
    
 Now, because $\alpha$ is almost finite in measure, for $K'$ and $\delta$ above, there exists an open castle $\CC=\{(S_i, B_i): i\in I\}$ such that 
    \begin{enumerate}[label=(\roman*)]
         \item each shape $S_i$ is $(K', \delta)$-F{\o}lner and thus $(K, \epsilon)$-F{\o}lner;
         \item the diameter $\diam(sB_i)<\delta\leq \epsilon$ for any $s\in S_i$ and $i\in I$;
         \item we have $\sup_{\mu\in M_\Gamma(X)}\mu(X\setminus \bigsqcup_{i\in I}S_iB_i)<\delta$.
     \end{enumerate}
 Now, for each $i\in I$, by our choice of $K'$ and $\delta$, each $(K', \delta)$-F{\o}lner shape $S_i$ contains a subset $S'_i\subset S_i$ such that 
 \[\frac{1}{n+1}|S_i|< |S'_i|< \frac{1}{n}|S_i|.\]
 Write $O=\bigsqcup_{i\in I}S'_iB_i$ for simplicity and for any $\mu\in M_\Gamma(X)$, one has 
 \[\mu(O)=\sum_{i\in I, s\in S'_i}\mu(sB_i)> \frac{1}{n+1}\sum_{i\in I, s\in S_i}\mu(sB_i)=\frac{1}{n+1}\mu(\bigsqcup_{i\in I}S_iB_i)> \frac{1-\delta}{n+1}.\]
 We write $F=X\setminus \bigsqcup_{i\in I}S_iB_i$, which is a closed set satisfying $\mu(F)<\delta$ for any $\mu\in M_\Gamma(X)$. Then \cite[Lemma 3.2]{M1} implies that there exists an open set $U$ containing $F$ such that $\sup_{\mu\in M_\Gamma(X)}\mu(U)<\delta$. This further implies that
 
 \[\mu(U)<\delta< \frac{1-\delta}{n+1}<\mu(O)\]
 for any $\mu\in M_\Gamma(X)$ by our choice of $\delta$. Now because $\alpha$ is assumed to have dynamical comparison, one has \[X\setminus \bigsqcup_{i\in I}S_iB_i=F\prec O=\bigsqcup_{i\in I}S'_iB_i.\]
 Then take properties (i) and (ii) for $\{(S_i, B_i): i\in I\}$ above into consideration, one has $\alpha$ is almost finite.
 \end{proof}

 Even though we do not use it in this paper, it is worth mentioning that the converse direction of Proposition \ref{rmk: af in measure plus comparison equal af} is also true. The proof is similar to \cite[Theorem 9.2]{D}, which means that assumptions of minimality and freeness there can be dropped. This has been used implicitly in \cite{K-S}. Also for the convenience of the reader, we provide a proof here. 

 \begin{prop}\label{prop: af imply af in measure and comparison}
     Let $\alpha: \Gamma\curvearrowright X$ be an action of an infinite amenable group $\Gamma$ on a compact metrizable space $X$. Suppose $\alpha$ is almost finite. Then it is almost finite in measure and has dynamical comparison. 
     \end{prop}
\begin{proof}
    Suppose $\alpha$ is almost finite. It is direct to see that $\alpha$ is almost finite in measure. This is mainly because for any closed set $F$, non-empty open set $O$ and $\mu\in M_\Gamma(X)$, the condition $F\prec O$ implies $\mu(F)\leq \mu(O)$.
    
    Thus, it suffices to show $\alpha$ has dynamical comparison. This means that it suffices to establish $A\prec B$ whenever $A$ is a non-empty closed set and $B$ is a non-empty open set satisfying $\mu(A)<\mu(B)$ for any $\mu\in M_\Gamma(X)$.
    
    First, \cite[Lemma 3.3]{D} implies that there exists a $\delta>0$ such that the sets $B_-=\{x\in X: d(x, X\setminus B)>\eta\}$ and $A_+=\{x\in X: d(x, A)\leq \eta\}$ satisfying $\mu(A_+)+\eta\leq \mu(B_+)$ for any $\mu\in M_\Gamma(X)$. Then as in \cite[Theorem 9.2]{D}, there exists a finite $K\subset \Gamma$ and a $\delta>0$ such that for any $x\in X$ and $(K, \delta)$-F{\o}lner set $F\subset \Gamma$, one has
    \begin{align*}\label{eq}
        \frac{1}{|F|}\sum_{s\in F}1_{A_+}(sx)+\frac{\eta}{2}\leq \frac{1}{|F|}\sum_{s\in F}1_{B_-}(sx),\tag{$\star$}   
        \end{align*}
    where $1_{A_+}$ and $1_{B_-}$ denote characteristic functions on $A_+$ and $B_-$, respectively.
Now, since $\alpha$ is almost finite, for the $K, \delta, \eta$ above and an integer $n>3/\eta$, there exists an open castle $\CC=\{(S_i, V_i): i\in I\}$ such that 
 \begin{enumerate}[label=(\roman*)]
         \item each shape $S_i$ is $(K, \delta)$-F{\o}lner;
         \item the diameter $\diam(sV_i)<\eta$ for any $s\in S_i$ and $i\in I$;
         \item for each $i\in I$ there exists $S'_i\subset S_i$ with $|S'_i|<(1/n)|S_i|$ such that one has $X\setminus \bigsqcup_{i\in I}S_iV_i\prec \bigsqcup_{i\in I}S'_iV_i$.
     \end{enumerate}
     For $i\in I$, define $S_{i, 1}=\{s\in S_i: sV_i\cap A\neq \emptyset\}$ and $S_{i, 2}=\{s\in S_i: sV_i\cap B_-\neq \emptyset\}$. By our choice of $\delta$, one has that $sV_i\subset A_+$ for any $s\in S_{i, 1}$ and $sV_i\subset B$ for any $s\in S_{i, 2}$.
     
     In addition, choosing an $x_i\in V_i$ for each $i\in I$, the inequality (\ref{eq}) above implies that 
    \[\frac{|S_{i, 1}|}{|S_i|}+\frac{\eta}{2}\leq \frac{|S_{i, 2}|}{|S_i|}\]
   for each $i\in I$ and our choice of $n>2/\eta$ and condition (iii) further imply 
   \[|S_{i,1}|+|S'_i|<|S_{i ,2}|.\]
   Therefore, for each $i\in I$, it allows us to define two injective maps $\varphi_i: S_{i, 1}\to S_{i, 2}$ and $\psi_i: S'_i\to S_{i,2}$ with disjoint ranges, i.e., $\varphi_i(S_{i, 1})\cap \psi_i(S'_i)=\emptyset$.
   
   For $X\setminus \bigsqcup_{i\in I}S_iV_i\prec \bigsqcup_{i\in I}S'_iV_i$, the same argument in the last paragraph of \cite[(i)$\Rightarrow$(ii) of Theorem 9.2]{D} shows that there exists an open cover $\CW$ of $X\setminus \bigsqcup_{i\in I}S_iV_i$ and $t_W\in \Gamma$ for each $W\in \CW$ such that $\{t_WW: W\in \CW\}$ is disjoint and each $t_WW$ is contained in an $sV_i$ for some $i\in I$ and $s\in S'_i$.  
Now, one has  
\[A\subset (\bigsqcup_{i\in I}S_{i, 1}V_i)\cup (X\setminus \bigsqcup_{i\in I}S_iV_i)\subset (\bigsqcup_{i\in I, s\in S_{i, 1}}sV_i)\cup \bigcup\CW\]
and the disjoint union
\[\bigsqcup_{i\in I, s\in S_{i, 1}}(\varphi_i(s)s^{-1})sV_i\sqcup \bigsqcup_{i\in I, s\in S'_i}\bigsqcup\{\psi_i(s)s^{-1}t_WW: t_WW\subset sV_i\}\]
is a subset of $\bigsqcup_{i\in I}S_{i, 2}V_i\subset B$. Thus, we have $A\prec B$.
\end{proof}

The \textit{small boundary property}, abbreviated as the SBP as usual, for an action, originated in the work \cite{S-W} of Shub and Weiss and played an important role in later work on the mean dimension of actions. (see, e.g., \cite{L-W}). 

\begin{defn}\label{defn: SBP}
    An action $\Gamma\curvearrowright X$ is said to have the small boundary property (the SBP) if for any $x\in X$ and open set $U\ni x$, there exists an open neighborhood $V$ of $x$ such that $x\in V\subset U$ and $\mu(\partial V)=0$ holds for any $\mu\in M_\Gamma(X)$.  
\end{defn}

It was proved in \cite{K-S} in the context of free actions that the SBP is equivalent to almost finiteness in measure.

\begin{thm}\cite[Theorem 5.6]{K-S}
    Let $\alpha: \Gamma\curvearrowright X$ be a free action. Then $\alpha$ has the SBP if and only if $\alpha$ is almost finite in measure.
\end{thm}

We finally recall some applications of almost finiteness and almost finiteness in measure in the structure theory of $C^*$-algebras. We denote by $\CZ$ the \textit{Jiang-Su} algebra. A $C^*$-algebra $A$ is said to be $\CZ$-stable if $A\otimes \CZ\simeq A$. We refer to \cite{C-E-T-W} for the definition of \textit{uniform property} $\Gamma$ and its application to the Toms-Winter conjecture.

\begin{thm}\cite[Theorem 12.4]{D}\label{thm: af Z stability}
  Let $\alpha: \Gamma\curvearrowright X$ be a minimal free almost finite action. Then the crossed product $C(X)\rtimes_r \Gamma$ is $\CZ$-stable.  
\end{thm}

\begin{thm}\cite[Theorem 9.4, Corollary 9.5]{K-S}\label{thm: af in measure Gamma}
Let $\alpha: G\curvearrowright X$ be a minimal free action that is almost finite in measure. Then the crossed product $C(X)\rtimes_r G$ has uniform property $\Gamma$ and thus satisfies the Toms-Winter conjecture.
\end{thm}

\begin{rmk}\label{ess free}
   In the statement of Theorem \ref{thm: af in measure Gamma}, we have replaced the original assumption on the SBP in \cite[Theorem 9.4]{K-S} by almost finiteness in measure for convenience.  This implies that 
   the assumption of freeness of the actions in Theorem both \ref{thm: af Z stability} and \ref{thm: af in measure Gamma} can be dropped. On the other hand, the freeness of the action is also used implicitly in the proof \cite[Theorem 9.4]{K-S} that $\tau=\tau\circ E$ for any trace $
   \tau\in T(C(X)\rtimes \Gamma)$, where $E$ is the canonical conditional expectation. But this is still true for essentially free actions. See e.g. \cite[Theorem 2.7]{KTT}, \cite[Corollary 1.2]{Ne} and \cite[Theorem A]{LZ24} for the groupoid case. We will show below in Proposition \ref{prop: af in measure imply ess free} that essential freeness is a necessary condition of almost finiteness in measure.
   
   See also more general versions of these theorems in \cite{M-W} and \cite{M} in the framework of \'{e}tale groupoids.\end{rmk}

To end this section, we record the following result by many hands on the classification of $C^*$-algebras using Elliott invariant, i.e., the K-theory together with tracial information. See, e.g., \cite{EGLN}, \cite{GLN}, \cite{TWW}, \cite{CETWW} and \cite{CGSTW}.

\begin{thm}\label{thm: classification}
    Let $A$ be a unital simple separable nuclear stably finite $\CZ$-stable $C^*$-algebra satisfying the UCT.  Then $A$ is classifiable by its Elliott invariant.
    \end{thm}

\section{Almost finiteness and Almost finiteness in measure}
In this section, we prove the theorems mentioned in the introduction. We first recall the construction in \cite[Theorem 5.5]{K-S}, which states that for any $\alpha: \Gamma\curvearrowright X$ with the SBP, there is a zero-dimensional extension  $\pi: (Z, \gamma)\to (X, \alpha)$ that is \textit{measure-isomorphic over singleton fibers} in the sense of \cite[Definition 4.2]{K-S}.
For simplicity, we refer the reader to \cite[Section 5]{K-S} for all necessary related details on the construction of $(Z,\gamma)$, which has a similar flavor to odometers. Recall that a closed set $C$ is said to be \textit{regular closed} if $C=\overline{\operatorname{int}(C)}$.
Based on \cite[Theorem 5.5]{K-S}, we have the following observation.
 \begin{prop}\label{prop: minimal SBP extension}
   Let  $\alpha: \Gamma\curvearrowright X$ be a minimal action on an infinite compact metrizable space $X$ with the SBP. Then there is a minimal extension $\pi: (Z, \gamma)\to (X, \alpha)$, which is measure-isomorphic over singleton fibers such that $Z$ is the Cantor set. 
 \end{prop}
 \begin{proof}
Let $\{F_n: n\in \N\}$ be an increasing sequence for $\Gamma$ with $\bigcup_{n\in \N}F_n=\Gamma$ and $e\in F_n=F^{-1}_n$. We refer to \cite[Theorem 5.5]{K-S} for all notions, such as nested regular closed partitions (in the sense of \cite[Definition 5.2]{K-S}) $\mathscr{L}_n$ of $X$ constructed from $F_n$ for $n\in \N$  and necessary details of the construction of $Z$ and $\gamma$ from this sequence of partitions $\mathscr{L}_n$. The fact we will use below is that for any $n\in\N$ and $s\in F_n$, the partition $\alpha_s(\mathscr{L}_{n+1})$ refines $\mathscr{L}_n$ in the sense of \cite[Remark 5.3]{K-S}

We first note that if the action $\gamma:\Gamma\curvearrowright Z$ is minimal, then $Z$ has to be a Cantor set. Suppose not, let $z\in Z$ be an isolated point, i.e., $\{z\}$ is open in $Z$. Then, the minimality of $\gamma$ implies that for any $z'\in Z$, there exists an $s\in \Gamma$ such that $sz'\in \{z\}$, which implies that $\Gamma\cdot \{z\}=Z$. This  entails that $Z$ is finite by compactness, which is a contradiction to the fact that $Z$ is infinite as there is a subjective map $\pi: Z\to X$.
 
Thus, it suffices to show the action $\gamma: \Gamma\curvearrowright Z$ in \cite[Theorem 5.5]{K-S} is minimal whenever the action $\alpha$ is minimal. To this end,   let $z=(C_n)_n\in Z$ with $\bigcap_{n\in\N}C_n=\{x\}$. Let $U$ be an open set in $Z$. Without loss of any generality, one may assume \[U=\{(A_n)_n\in Z: A_n=D_n\text{ for all }1\leq n\leq k\}\]
     for some decreasing regular closed sets $D_1\supset D_{2}\supset\dots \supset D_k$ such that each $D_n\in \mathscr{L}_n$. Now since $\alpha$ is minimal, there exists an $s\in \Gamma$ such that $s\cdot x\in \operatorname{int}(D_k)\subset D_k$. This can be done because each partition $\mathscr{L}_n$ of $X$ is a regular closed partition and 
     the interior $\operatorname{int}(D_k)$ of $D_k$ is thus non-empty. Therefore, let $k_0\in \N$ be large enough such that $s\in F_{k_0}$ and $k_0> k$.  Note that $\alpha_s(\mathscr{L}_{k_0+1})$ refines $\mathscr{L}_{k_0}$ and thus refines $\mathscr{L}_k$. This entails that $\alpha_s(C_{k_0+1})\subset D_k$ because there is only one $D\in \mathscr{L}_k$ such that $\alpha_s(C_{k_0+1})\subset D$ by \cite[Remark 5.2]{K-S}. By definition of $\gamma$, one has \[\gamma(z)=(r^s_n(\alpha_s(C_{n+1})))_{n\geq k_0},\] where $r^s_n: \alpha_s(\mathscr{L}_{n+1})\to \mathscr{L}_n$ is the \textit{refinement map} in the sense of \cite[Remark 5.2]{K-S}.  One then has that the first $k$ coordinates of $\gamma(z)$ are exactly $D_n$ for $1\leq n\leq k$, which implies that $\gamma(z)\in U$. Thus $\gamma: \Gamma\curvearrowright Z$ is minimal.
      \end{proof}

      The same proof of \cite[Theorem 5.6]{K-S} yields the following.

      \begin{cor}\label{cor: SBP af in measure}
        Let  $\alpha: \Gamma\curvearrowright X$ be a minimal action on the space $X$ with the SBP. Suppose the minimal Cantor extension $\gamma: \Gamma\curvearrowright Z$ of $\alpha$ described in Proposition \ref{prop: minimal SBP extension} is almost finite. Then  $\alpha$ is almost finite in measure. 
      \end{cor}

      It was proven in \cite[Theorem 2.10]{O-S} that any minimal action of the infinite dihedral group $D_\infty$ on the Cantor set is almost finite. Therefore, we have the following result as a warm-up, which has generalized \cite[Theorem 2.10]{O-S}.

      \begin{cor}\label{cor: D infty af}
          Every minimal $D_\infty$-action on an infinite compact metrizable space $X$ with the SBP is almost finite. 
      \end{cor}
          \begin{proof}
          Let $\alpha$ be a minimal 
 $D_\infty$-action on a compact metrizable space $X$ with the SBP. Then Proposition \ref{prop: minimal SBP extension} implies that there is a minimal Cantor extension $\pi: (Z, \gamma)\to (X, \alpha)$, which is measure-isomorphic over singleton fibers. It was proven in  \cite[Theorem 2.10]{O-S} that $\gamma$ is almost finite. Therefore, Corollary \ref{cor: SBP af in measure} implies that $\alpha$ is almost finite in measure. Finally, since $D_\infty$ is of polynomial growth, the action $\alpha$ has dynamical comparison by \cite[Theorem A]{Na}, and thus $\alpha$ is almost finite by Proposition \ref{rmk: af in measure plus comparison equal af}.
          \end{proof}
\begin{rmk}
As every minimal $D_\infty$-action on an infinite compact metrizable space $X$ with the SBP is almost finite, it follows that $C(X)\rtimes_r D_\infty$ is classifiable by its Elliott invariant and has nuclear dimension at most one by \cite[Corollary~9.10]{M-W}. Hence, we have extended \cite[Theorem~C]{BL} from finite covering dimension to the small boundary property.
\end{rmk}

We now focus on a more general setting that our acting group $\Gamma$ is locally finite-by-$\Z$ or locally finite-by-$D_\infty$ groups. We begin with the following two general facts.

\begin{prop}\label{prop: extension ess free}
          Let $\Gamma$ be a countable discrete group and $\pi: (Z, \gamma)\to (X, \alpha)$ is an extension of $\Gamma$-dynamical systems such that $M_\Gamma(Z)\neq \emptyset$. Suppose $\alpha: \Gamma\curvearrowright X$ is essentially free. Then so is $\gamma: \Gamma\curvearrowright Z$.
      \end{prop}
      \begin{proof}
      Let $g\in \Gamma\setminus \{e_\Gamma\}$ and $\mu\in M_\Gamma(Z)$. Denote by $F=\{z\in Z: gz=z\}$ the fixed point set for $g$. Then observe that $F\subset \pi^{-1}(\pi(F))$ and thus $\mu(F)\leq(\pi^*\mu)(\pi(F))$, where $\pi^*\mu$ is the push-forward measure of $\mu$. Then note that 
      $\pi(F)\subset \{x\in X: gx=x\}$, which implies that 
      \[\mu(F)\leq (\pi^*\mu)(\{x\in X: gx=x\})=0\]
      because $\alpha$ is essentially free. Thus $\gamma: \Gamma\curvearrowright X$ is also essentially free.
      \end{proof}

It was shown in \cite[Remark 2.4]{O-S} and \cite[Lemma 2.2]{Joseph} that almost finiteness of an action $\alpha$ implies that $\alpha$ is necessarily essentially free. We show here that the result actually holds for actions that are almost finite in measure.

\begin{prop}\label{prop: af in measure imply ess free}
    Suppose $\alpha: \Gamma\curvearrowright X$ is an action that is almost finite in measure. Then the action $\alpha$ is essentially free.
\end{prop}
\begin{proof}
    Let $\mu\in M_\Gamma(X)$, $\epsilon>0$ and $g\in \Gamma\setminus\{e_\Gamma\}$. Since $\alpha$ is almost finite in measure, there exists an open castle $\{(S_i, B_i): i\in I\}$ such that 
    \begin{enumerate}[label=(\roman*)]        \item each $S_i$ is a F{\o}lner set such that $|S_i\setminus g^{-1}S_i|<(\epsilon/2)|S_i|$;  
        \item we have $\sup_{\mu\in M_\Gamma(X)}(X\setminus \bigsqcup_{i\in I}S_iB_i)<\epsilon/2$.
    \end{enumerate}
    Now, denote by $F_g=\{x\in X: gx=x\}$ the set of fixed points by $g$ and for each $i\in I$, we write $T_i=g^{-1}S_i\cap S_i$ for simplicity, which satisfies $|T_i|\geq (1-\epsilon/2)|S_i|$. Note that $F_g\cap (\bigsqcup_{i\in I}T_iB_i)=\emptyset$.  Suppose this is not the case. Then there exists an $i\in I$ and an $s\in T_i$ such that $x\in sB_i\cap F_g$. On the other hand, note that $gs\in S_i$ by the definition of $T_i$, which implies that $x=gx\in gsB_i$ as well. This is a contradiction because $s\neq gs$. Therefore, one necessarily has 
    \[F_g\subset X\setminus \bigsqcup_{i\in I}T_iB_i= (\bigsqcup_{i\in I}(S_i\setminus T_i)B_i)\sqcup (X\setminus \bigsqcup_{i\in I}S_iB_i),\]
    which entails that $\sup_{\mu\in M_\Gamma(X)}\mu(F_g)<\epsilon$. Since $\epsilon$ is arbitrary, one actually has $\mu(F_g)=0$ for any $\mu\in M_\Gamma(X)$ and thus  $\alpha$ is essentially free.
\end{proof}

Suppose $\Gamma$ is an amenable group satisfying
\[\begin{tikzcd}
0\arrow[r] & H \arrow[r, "i"] & \Gamma\arrow[r, "\rho"] & G\arrow[r] & 0,
\end{tikzcd}\]
in which $H$ is a locally finite group. Recall that if $H$ is finitely generated then $H$ is a finite group. Otherwise, there exists a strict increasing sequence of finite groups $e_\Gamma\in F_1\leq F_2\leq\dots$ such that $H=\bigcup_{n\in \N}F_n$. Note that $\{F_n: n\in \N\}$ actually forms a F{\o}lner sequence of $H$, which can be used to describe F{\o}lner sets in $\Gamma$ in the sense of the following lemma. The proof is based on an elementary combinatorial argument. We include the proof here for completeness.

\begin{lem}\label{lem: Folner set in extension}
    Let $\Gamma$ be an amenable group and
    \[\begin{tikzcd}
0\arrow[r] & H \arrow[r, "i"] & \Gamma\arrow[r, "\rho"] & G\arrow[r] & 0
\end{tikzcd}\]
the group extension above, in which $H=\bigcup_{n\in \N}F_n$ for a non-decreasing sequence of finite  groups $\{F_n: n\in \N\}$. Then for any finite set $K\subset \Gamma$ and $\epsilon>0$, there exists a $(\rho(K), \epsilon/|K|)$-F{\o}lner set $S$ in $G$ and large enough $n\in \N$ such that, if we denote by $\tilde{S}$ a lift of $S$ and write 
$A=\tilde{S}\cdot F_n$, then one has  $|K\cdot A\setminus A|<\epsilon |A|$.\end{lem}
\begin{proof}
    Write $\delta=\epsilon/|K|$ for simplicity. Then choose a $(\rho(K), \delta)$-F{\o}lner set $S$ in $G$. Then one may write $S=\{g_iH: i\in I\}$ for some $g_i\in \Gamma$ and a finite index set $I$. Denote by $\tilde{S}=\{g_i: i\in I\}$, the representative set of $S$. Then choose large enough $n\in \N$ such that $\tilde{S}^{-1}K\tilde{S}\cap H\subset F_n$. Write $F=F_n$ for simplicity and observe that the choice of $F$ implies that
    \[sg_iH=g_jH \Longleftrightarrow sg_iF=g_jF\]
    for any $i, j\in I$ and $s\in K$. This entails that 
    \[|\{sg_iF: i\in I\}\setminus \{g_iF: i\in I\}|=|\{sg_iH: i\in I\}\setminus \{g_iH: i\in I\}|=|\rho(s)S\setminus S|\]
    for any $s\in K$.
    
    Now look at $A=\tilde{S}\cdot F=\bigsqcup_{i\in I}g_iF$, whose cardinality satisfies $|A|=|S|\cdot|F|$. For any $s\in K$, write
    \[sA\setminus A=(\bigsqcup_{i\in I}sg_iF)\setminus (\bigsqcup_{i\in I}g_iF)\]
    and because $F$-cosets are equal or disjoint, one has
    \[|sA\setminus A|\leq |\rho(s)S\setminus S|\cdot|F|=\delta|S|\cdot|F|=\delta|A|.\]
    Therefore, one has $|K\cdot A\setminus A|\leq |K|\delta\cdot |A|=\epsilon|A|.$
      \end{proof}

\begin{rmk}
We remark that the above Lemma \ref{lem: Folner set in extension} is a special case of the following known result, whose proof is more complicated. Let \[\begin{tikzcd}
0\arrow[r] & H \arrow[r, "i"] & \Gamma\arrow[r, "\rho"] & G\arrow[r] & 0
\end{tikzcd}\]
be the extension of countable discrete groups $G, H$ and $\Gamma$. Suppose $\{F_n: n\in \N\}$ and $\{E_m: m\in \N\}$ are F{\o}lner sequences for $H$ and $G$, respectively. For each $m$, denote by $E'_m$ a lift of $E_m$.
Then for any finite set $S\subset \Gamma$, $\epsilon>0$, there exists a set $A=E'_m\cdot F_n\subset \Gamma$ such that $|S\cdot A\setminus A|<\epsilon |A|$.

\end{rmk}
Now, we look at the orbit space $X/H$, equipped with an induced natural action $\beta: G\simeq \Gamma/H\curvearrowright X/H$, defined in the way that $\beta(sH)(Hx)=H\alpha_s(x)$. Since $H$ is normal, it is direct to see $\beta$ is well-defined. Moreover, if we denote by $\pi: X\to X/H$ the canonical quotient maps, then observe the action $\beta$ is compatible with $\alpha$, $\pi$ and group quotient homomorphism $\rho$ in the sense that $\pi(\alpha(s)x)=\beta(\rho(s))(\pi(x))$.

\begin{rmk}\label{rmk: basic property of X/H}
We recall some basic properties of the quotient map $\pi$ and the action $\beta$. 
\begin{enumerate}[label=(\roman*)]
\item The space $X/H$ is  infinite whenever $\alpha$ is minimal. Suppose not, there are only finitely many $H$-orbits in $X$. This is a contradiction to the fact that $X$ is an infinite compact space and the minimality of $\alpha$.
\item The map $\pi$ is open, which comes from the fact $\pi^{-1}(\pi(U))=HU$.
\item The space $X/H$ is zero-dimensional if $X$ is zero-dimensional. This is because (ii) implies that $\{\pi(A): A\text{ is clopen in } X\}$ is a basis for $X/H$.
\item The action $\beta$ is minimal whenever $\alpha$ is minimal. Indeed, let $A$ be a non-empty clopen set in $X$. observe that $G\cdot \pi(A)=\pi(\Gamma\cdot A)=X/H$ because $\Gamma\cdot A=X$ in this case. 
    \end{enumerate}
\end{rmk}

Moreover, the following needs a more complicated argument.

\begin{prop}\label{prop: ess free permanance}
Let \[\begin{tikzcd}
0\arrow[r] & H \arrow[r, "i"] & \Gamma\arrow[r, "\rho"] & G\arrow[r] & 0
\end{tikzcd}\]
be the extension of countable discrete groups $G, H$ and $\Gamma$.
    Suppose $\alpha: \Gamma\curvearrowright X$ is an essentially free action such that $X/H$ is Hausdorff. Let $\beta: G\curvearrowright X/H$ be the induced action mentioned above. Then $\beta$ is also essentially free. 
\end{prop}
\begin{proof}
   Write $Y=X/H$ for simplicity and let $\mu\in M_\Gamma(X)$. Define  a ``push-forward'' probability measure $\pi^*\mu$ on $Y$ by $\pi^*\mu(A)=\mu(\pi^{-1}(A))$. Because $\rho: \Gamma\to G$ is a quotient homomorphism, for any $s\in \Gamma$ and set $A$ in $Y$, one has
   \[\pi(x)\in \rho(s)A\Longleftrightarrow \rho(s^{-1})\pi(x)\in A\Longleftrightarrow \pi(s^{-1}x)\in A,\]
   which implies that $\pi^{-1}(\rho(s)A)=s\pi^{-1}(A)$. Thus, one has
 $\pi^*\mu\in M_G(Y)$. Then observe that $\pi^*: M_\Gamma(X)\to M_G(Y)$ is subjective. This is mainly due to the Hahn-Banach theorem exactly as in the remark after \cite[Proposition 4.3]{K-S}. To be more specific, let $\nu\in M_G(Y)$, which induces a $G$-invariant state $\sigma\in C(Y)^*$. On the other hand, $C(Y)$ can be viewed as a subspace in $C(X)$ via the embedding $f\mapsto f\circ \pi$. Using the Hahn-Banach theorem, one may extend $\sigma$ to a state $\lambda\in C(X)^*$. For any $s\in \Gamma$ and $f\in C(Y)$, note that
 \[(s\cdot(f\circ\pi))(x)=f(\pi(s^{-1}x))=f(\rho(s)^{-1}\pi(x))=((\rho(s)\cdot f)\circ\pi)(x),\]
 which implies that $\lambda(s\cdot (f\circ\pi))=\sigma(\rho(s)\cdot f)=\sigma(f)$. Let $\{F_n: n\in \N\}$ be a F{\o}lner sequence of $\Gamma$. Then any weak*-cluster point of $g\mapsto |F_n|^{-1}\sum_{s\in F_n}\lambda(s\cdot g)$ in $C(X)^*$ yields a $\Gamma$-invariant state $\tilde{\lambda}\in C(X)^*$ and thus a $\Gamma$-invariant probability measure $\mu$ in $M_\Gamma(X)$ by Riesz representation theorem.  The construction of $\tilde{\lambda}$ implies that $\tilde{\lambda}(f\circ \pi)=\sigma(f)$ for any $f\in C(Y)$, which implies $\pi^*\mu=\nu$.

 Now, let $\rho(s)$ be a nontrivial element in $G$ and let $A=\{\pi(x): \rho(s)\pi(x)=\pi(x)\}$ be the fixed point set of $\rho(s)$ in $Y$. Observe that 
 \[\pi^{-1}(A)=\{x\in X: Hsx=Hx\}=\bigcup_{h\in H}\{x\in X: hsx=x\}.\]
 Then since $hs$ is not a trivial element in $\Gamma$, one has $\mu(\{x\in X: hsx=x\})=0$ for any $\mu\in M_\Gamma(X)$ because $\alpha: \Gamma\curvearrowright X$ is essentially free. Now, since $H$ is countable, one has 
 \[\pi^*\mu(A)\leq \sum_{h\in H}\mu(\{x\in X: hsx=x\})=0.\]
 This implies that $\beta:G\curvearrowright Y$ is also essentially free.
\end{proof}

\begin{prop}\label{prop: af in measure permanance}
Let \[\begin{tikzcd}
0\arrow[r] & H \arrow[r, "i"] & \Gamma\arrow[r, "\rho"] & G\arrow[r] & 0
\end{tikzcd}\]
be the extension of countable discrete groups $G, H$ and $\Gamma$ in which $H$ is locally finite.
Suppose $\alpha: \Gamma\curvearrowright X$ is a minimal essentially free action on a zero-dimensional space $X$ such that $X/H$ is Hausdorff. Suppose the induced action $\beta: G\curvearrowright X/H$ defined above is almost finite. Then $\alpha$ is almost finite in measure.
\end{prop}
\begin{proof}
Write $Y=X/H$ for simplicity. Let $e\in K\subset \Gamma$ be a finite set and $\epsilon>0$. Since $K$ is finite and $H$ can be written as $H=\bigcup_{n\in \N}F_n$ for a non-decreasing sequence of finite groups $F_n$.

Since $\beta$ is almost finite and $Y$ is zero-dimensional, one may find a finite castle $\CC=\{(S_i, V_i): i\in I\}$ such that 
\begin{enumerate}[label=(\roman*)]    \item each $S_i$ consists $e_{G}$ and $S_i$ is a $(\rho(K), \epsilon/|K|)$-F{\o}lner set in $G$;
\item each $V_i$ is a clopen set in $Y$;
\item we have $Y=\bigsqcup_{i\in I}S_iV_i$.
\end{enumerate}
Let $i\in I$. We enumerate $S_i$ by $\{g^i_jH: j\in J_i\}$ for some finite index set $J_i$, in which each $g^i_j\in \Gamma$ is a fixed representative for the coset contained in $S_i$. We also choose some $g^i_j=e_\Gamma$ for $e_G=H\in G$. Choose a proper finite subgroup $F\leq H$ and Lemma \ref{lem: Folner set in extension} implies that $\tilde{S}_i=\bigsqcup_{j\in J_i}g^i_jF\subset \Gamma$ is a $(K, \epsilon)$-F{\o}lner set. We then build castles with F{\o}lner shapes $\tilde{S}_i$ in $X$.

For simplicity, we denote by $W_i=\pi^{-1}(V_i)$, which is a $H$-invariant clopen set in $X$. We first claim that for $i, i'\in I$ and $j\in J_i$ and $j'\in J_{i'}$ satisfying $(i, j)\neq (i', j')$, one has $g^i_jW_i$ is disjoint from $g^{i'}_{j'}W_{i'}$. Suppose not, let $x\in g^i_jW_i\cap g^{i'}_{j'}W_{i'}$. Then one has $\pi(x)\in \rho(g^i_j)V_i\cap \rho(g^{i'}_{j'})V_{i'}$, which is a contradiction to the fact that $\rho(g^i_j)V_i$ and $\rho(g^{i'}_{j'})V_{i'}$ are different levels in the castle $\CC$. Then observe that actually $X=\bigsqcup_{i\in I}\bigsqcup_{j\in J_i}g^i_jW_i$.

Then, since each $W_i$ is $H$-invariant and thus $F$-invariant, for any $x\in W_i$, choose a clopen neighborhood $x\in A_x\subset W_i$ and the family $\{F\cdot A_x: x\in W_i\}$ form an open cover of $W_i$, which yields a finite subcover $\{F\cdot A^i_k: k=1,\dots, n_i\}$ of $W_i$. Now define $B^i_1=A^i_1$ and $B^i_{k}=A^i_k\setminus \bigcup_{1\leq j<k}FA^i_j$ for $1<k\leq n$. Then, one has $FB^i_k\cap FB^i_l=\emptyset$ and $\bigsqcup_{k=1}^nFB^i_k=W_i$. 

 Denote by $T=\{x\in X: sx=x\text{ for some }s\in F\}$, which is a closed set. Since $\alpha$ is essentially free, one actually has $\sup_{\mu\in M_\Gamma(X)}\mu(T)=0$. Then, for the $\epsilon$ above, \cite[Lemma 3.2]{M1} implies that there is a $\delta>0$ such that
 \[\sup_{\mu\in M_\Gamma(X)}\mu(\{x\in X: d(x, T)\leq \delta\})\leq \epsilon/(\sum_{i\in I}|\tilde{S}_i|).\]
 Therefore, there exists a clopen set $N$ by compactness of $T$ such that 
 \[T\subset N\subset \{x\in X: d(x, T)\leq \delta\}\]
 and thus also satisfying $\sup_{\mu\in M_\Gamma(X)}\mu(N)\leq \epsilon/(\sum_{i\in I}|\tilde{S}_i|)$.
 
 For each $i\in I$ and $k\leq n_i$, define $C^i_k=B^i_k\setminus N$, which is a clopen set consisting of no fixed points of $F$. Therefore,  via further decomposition of $C^i_k$ if necessary, we may assume $(F, C^i_k)$ is a tower. Then, by our construction above, the collection of all $(\bigsqcup_{j\in J_i}g^i_jF, C^i_k)$ are disjoint towers for $i\in I$ and $1\leq k\leq n_i$, i.e., the collection $\{(\tilde{S}_i, C^i_k): 1\leq k\leq n_i, i\in I\}$ is a clopen castle.

 Finally, observe that 
 \[X\setminus \bigsqcup_{i\in I}\bigsqcup_{1\leq k\leq n_i}\tilde{S}_iC^i_k\subset \bigcup_{s\in \bigcup_{i\in I}\tilde{S}_i}sN,\]
 which implies that
 \[\mu(X\setminus \bigsqcup_{i\in I}\bigsqcup_{1\leq k\leq n_i}\tilde{S}_iC^i_k)\leq \sum_{s\in \bigcup_{i\in I}\tilde{S}_i}\mu(sN)\leq \epsilon\]
  for any $\mu\in M_\Gamma(X)$. Recall that each shape $\tilde{S}_i$ is $(K, \epsilon)$-F{\o}lner. Thus $\alpha$ is almost finite in measure.
\end{proof}

As an application, we have arrived the following main technical result.
\begin{prop}\label{thm: af in measure}
Let
\[\begin{tikzcd}
0\arrow[r] & H \arrow[r, "i"] & \Gamma\arrow[r, "\rho"] & G\arrow[r] & 0
\end{tikzcd}\]
be the extension sequence of countable discrete groups $G, H$ and $\Gamma$
in which $H$ is locally finite and $G$ is either the integer group $\Z$ or the infinite dihedral group  $D_\infty$.
Suppose that for any minimal essentially free action $\gamma: \Gamma\curvearrowright Z$ on the Cantor set $Z$, the orbit space $Z/H$ is Hausdorff. Let $X$ be a compact metrizable space and $\alpha: \Gamma\curvearrowright X$ a minimal action. Then $\alpha$ is almost finite in measure if and only if it is essentially free and has the small boundary property.
\end{prop}
\begin{proof}
    Suppose $\alpha$ is almost finite in measure. Then $\alpha$ is essentially free by Proposition \ref{prop: af in measure imply ess free} above and has the SBP by \cite[Theorem 5.5]{K-S}. For the converse, because $\alpha$ has the SBP, \cite[Theorem 5.5]{K-S} implies that there exists an extension $\pi: (\Gamma, \gamma, Z)\to (\Gamma, \alpha, X)$, which is measure-isomorphic over singleton fibers such that $Z$ is zero-dimensional. In addition, Propositions \ref{prop: minimal SBP extension} and \ref{prop: extension ess free}  show that $Z$ actually is a Cantor set and the extended system $\gamma: \Gamma\curvearrowright Z$ is still minimal and essentially free. Now look at the space $Z/H$, which is compact Hausdorff by the assumption. Moreover, the space $Z/H$ is second countable and zero-dimensional by Remark \ref{rmk: basic property of X/H}(iii). The induced action $\beta: G\curvearrowright Z/H$ is minimal by \ref{rmk: basic property of X/H}(iv). Since $G$ is either $\Z$ or $D_\infty$, the action $\beta: G
    \curvearrowright Z/H$ is almost finite by the classical Kakutani-Rokhlin partition for $\Z$-minimal actions on zero-dimensional spaces or Corollary \ref{cor: D infty af}, respectively.
    Then Proposition \ref{prop: af in measure permanance} entails that $\gamma: \Gamma\curvearrowright Z$ is almost finite in measure.  Then the same argument in \cite[Theorem 5.6]{K-S} shows that $\alpha:\Gamma \curvearrowright X$ is almost finite in measure.
    \end{proof}

It is well known that every virtually $\Z$ group $\Gamma$ satisfies the exact sequence of group extension
\[\begin{tikzcd}
0\arrow[r] & F \arrow[r, "i"] & \Gamma\arrow[r, "\rho"] & G\arrow[r] & 0
\end{tikzcd}\]
in which $F$ is a normal finite group and $G$ is either $\Z$ or $D_\infty$. By applying Theorem \ref{thm: af in measure} and Remark \ref{rmk: allosteric}, we also have the following.

\begin{cor}\label{cor: virtually cyclic af}
    Let $\alpha: \Gamma\curvearrowright X$ be a minimal topologically free action of a virtually $\Z$ group $\Gamma$ on an infinite compact metrizable space $X$ with the small boundary property.  Then $\alpha$ is almost finite.
\end{cor}
\begin{proof}
Because $F$ in the above extension sequence is finite, $Z/F$ is indeed always Hausdorff for any minimal essentially free Cantor action $\Gamma\curvearrowright Z$. Then, it follows from Remark \ref{rmk: allosteric} that the virtually $\Z$ group $\Gamma$ is non-allosteric. Therefore, actually, $\alpha$ is minimal and essentially free. Therefore, Proposition \ref{thm: af in measure} implies that $\alpha$ is almost finite in measure. Now, since the virtually $\Z$ group $\Gamma$ is of polynomial growth, the action $\alpha$ has dynamical comparison by \cite[Theorem A]{Na}. Therefore, $\alpha$ is almost finite by Proposition \ref{rmk: af in measure plus comparison equal af}.
\end{proof}

We remark that Corollary \ref{cor: virtually cyclic af} generalized the $D_\infty$ case in Corollary \ref{cor: D infty af} because all minimal $D_\infty$-action on an infinite compact space $X$ is automatically topologically free, proved by Jiang in \cite[Proposition 2.6]{J}. See also \cite[Proposition 2.8]{O-S}. Since $D_\infty$ is non-allosteric, any minimal action of $D_\infty$ is also essentially free. However, we provide a direct proof for this fact based on \cite[Proposition 2.6]{J}.

\begin{prop}\label{prop: D infty ess free}
Let $\alpha: D_\infty\curvearrowright X$ be a minimal action on an infinite compact Hausdorff space, then it is essentially free.    
\end{prop}
\begin{proof}
   Let $\mu\in M_{D_\infty}(X)$ and write $D_\infty=\Z\rtimes \Z_2$ with generators $s, t$, which generate $\Z$ and $\Z_2$, respectively. We may assume $\alpha$ is not free and in this case, the restriction action $\alpha|_\Z$ on $X$ is minimal and thus free (see \cite[Proposition 2.6]{J}). Therefore, for any $x\in X$, the stabilizer group $\operatorname{stab}(x)$ at $x$, if not trivial, is of the form $\{e, s^nt\}$ for some $n\in \Z$. Denote  by $X_n=\{x\in X: s^ntx=x\}$. Then for any $x\in s^kX_n$ ($k\neq 0$), one has $s^{-k}x\in X_n$, which implies that
   \[s^{-k}x=s^nts^{-k}x=s^ns^ktx=s^{n+k}tx\] and thus $x\in X_{2k+n}$ holds, which implies $s^kX_n\subset X_{2k+n}$. On the other hand, suppose $X_n\cap  X_{2k+n}\neq \emptyset$. Choose an $x\in X_n\cap  X_{2k+n}$ and observe that 
   \[s^{2k+n}x=s^{2k+n}s^ntx=s^ns^{2k+n}tx=s^nx,\]
   which entails $s^{2k}\in \operatorname{stab}(x)$. This is a contradiction to that $\alpha|_\Z$ is free. Thus, $X_n$ has to be disjoint from  $X_{2k+n}$, which implies that $X_n$ is disjoint from  $s^kX_n$ for any $k\neq 0$ because $s^kX_n\subset X_{2k+n}$. Therefore, $\{s^kX_n: k\in \Z\}$ is a disjoint family in $X$. Now one has
   \[\sum_{k\in \Z}\mu(s^kX_n)=\sum_{k\in \Z}\mu(X_n)\leq 1,\]
   which implies that $\mu(X_n)=0$. Thus, the action $\alpha$ is essentially free.
\end{proof}

However, it is not true for a general virtually $\Z$ group $\Gamma$ that all its minimal action is topologically free. For example, let $\Gamma=F\times \Z$, where $F$ is a non-trivial finite group and define an action $\alpha: \Gamma\curvearrowright X$ such that $\alpha|_\Z$ is minimal and $\alpha|_F$ is trivial.
On the other hand, we obtain the following more general result by applying group extension twice.

\begin{cor}\label{cor: af in measure more general}
     Let \[\begin{tikzcd}
0\arrow[r] & F \arrow[r, "i"] & \Gamma\arrow[r, "\rho"] & H\arrow[r] & 0
\end{tikzcd}\]
be an extension of countable discrete groups $F, H$, and $\Gamma$, in which $F$ is locally finite and $H$ is virtually $\Z$. Suppose $\alpha: \Gamma \curvearrowright X$ is a minimal action on 
the Cantor set $X$ such that $X/F$ is Hausdorff.  Then $\alpha$ is almost finite in measure if and only if it is essentially free.
     \end{cor}
     \begin{proof}
         This is a direct application of Propositions \ref{prop: ess free permanance}, \ref{prop: af in measure permanance}, and Corollary \ref{cor: virtually cyclic af} and the same argument record in Theorem \ref{thm: af in measure}. 
     \end{proof}

From Corollary \ref{cor: virtually cyclic af} and \ref{cor: af in measure more general}, we have the following applications to the structure theory of crossed product $C^*$-algebras by Theorem \ref{thm: af Z stability}, \ref{thm: af in measure Gamma}, \ref{thm: classification} and Remark \ref{ess free}.

\begin{cor}\label{cor: final}
\begin{enumerate}[label=(\roman*)]
    \item Let \[\begin{tikzcd}
0\arrow[r] & F \arrow[r, "i"] & G\arrow[r, "\rho"] & H\arrow[r] & 0
\end{tikzcd}\]
be an extension of countable discrete groups $F, H$, and $G$, in which $F$ is locally finite and $H$ is virtually $\Z$. Suppose $G \curvearrowright X$ is a minimal action on the Cantor set $X$ such that the quotient space $X/F$ is Hausdorff. Then $C(X)\rtimes_r G$ has uniform property $\Gamma$ and thus satisfies the Toms-Winter conjecture. 
    \item Suppose $\alpha: G\curvearrowright X$ is a minimal
    topologically free action of a 
    virtually $\Z$ group $\Gamma$ on an infinite compact metrizable space $X$ with the small boundary property, then $C(X)\rtimes_r G$ is $\CZ$-stable and thus classifiable by its Elliott invariant.    \end{enumerate}
    \end{cor}

\section{Examples from subshifts}\label{sec4}
Even though it follows from \cite{AS} that minimality and topologically freeness are sufficient and necessary conditions for the crossed product to be simple, we decide to provide certain concrete examples from subshift that our theorems could apply. 

Let $\alpha:\Gamma\curvearrowright X$ be a topological dynamical system on a compact space $X$ and $x\in X$. It is well-known that $\Gamma\curvearrowright \overline{\Gamma\cdot x}$ is minimal if and only if $x$ is \textit{almost periodic} in the sense that for any neighborhood $U$ of $x$, the recurrent set $R_{x, U, \Gamma}=\{g\in \Gamma: gx\in U\}$ is \textit{syndetic}, i.e., there exists a finite set $K\subset \Gamma$ such that $K\cdot R_{x, U, \Gamma}=\Gamma$. See, e.g., \cite[Proposition 7.13]{Kerr-L}. If the context of the acting group $\Gamma$ is clear, we usually write $R_{x, U}$ instead of $R_{x, U, \Gamma}$ for simplicity.

Let $A$ be a finite alphabet and $w$ an infinite word in $A^\Z$. Define a ``mirror'' word $\bar{w}\in A^\Z$ of $w$ by $\bar{w}(n)=w(-n)$ for any $n\in \Z$. 

\begin{rmk}\label{rmk: recurrent set}
    We remark that $w$ is almost periodic if and only if $\bar{w}$ is almost periodic. For any open neighborhood of $\bar{w}$ with the form $U=\{y\in A^\Z: y(i)=a_i, i\in I\}$, where $I$ is a finite subset of $\Z$, define $V=\{y\in A^\Z: y(-i)=a_i, i\in I\}$ as the mirror image of $U$. Let $n\in R_{w,V}$, which implies that  $w(-i-n)=a_i$ for any $i\in I$. Thus, one has $(-n\cdot \bar{w})(i)=\bar{w}(i+n)=w(-i-n)=a_i$ for any $i\in I$ and thus $-n\in R_{\bar{w}, U}$. This implies that $-R_{w, V}\subset R_{\bar{w}, U}$. Now if $w$ is almost periodic, i.e., $R_{w, V}$ is syndetic, then there is a finite set $K\subset \Z$ such that $K+R_{w, V}=\Z$, which entails that $-K-R_{w, V}=\Z$. Therefore, $R_{\bar{w}, U}$ is syndetic as well. Thus $\bar{w}$ is almost periodic. Finally, note that $\bar{\bar{w}}=w$, and therefore we have shown the claim. Moreover, the argument above actually shows $R_{w, V}=-R_{\bar{w}, U}$ holds.
    \end{rmk}

We say an almost periodic word $w\in A^\Z$ \textit{balanced} if for any open neighborhood $U$ of $w$, there is a syndetic set $P\subset R_{w, U}$ in $\Z$ such that $P=-P$.   Typical examples of balanced almost periodic words are Toeplitz words. See \cite[Section 3.1 and Proposition 5]{CM} and see \cite{CK} for concrete examples of Toeplitz words with various complexities. Now for $D_\infty=\Z\rtimes \Z_2=\langle s,t| t^2, tsts\rangle$ and a word  $w\in A^\Z$, we define an amplified word $\hat{w}\in A^{D_\infty}$ as follows. Write $D_\infty=\Z\sqcup \Z \cdot t$ and define $\hat{w}(n)=w(n)$ and $\hat{w}(nt)=\bar{w}(n)$ for $n\in \Z$, where $\bar{w}$ is the mirror word of $w$. 

\begin{lem}\label{lem: almost periodic}
    Suppose $w\in A^\Z$ is a balanced almost periodic word for the left shift action by $\Z$. Then its amplified word $\hat{w}\in A^{D_\infty}$ defined above is almost periodic for the left shift action by $D_\infty$. 
\end{lem}
\begin{proof}
    Note that $A^{D_\infty}=A^{\Z\sqcup \Z\cdot t}=A^\Z\times A^{\Z \cdot t }$. Now let $U, V$ be non-empty open sets in $A^\Z$ and $A^{\Z\cdot t}$ such that $\hat{w}\in U\times V$. It suffices to show $R_{\hat{w}, U\times V}$ is syndetic. Without loss of generality, one may assume $U=\{y\in A^\Z: y(i)=a_i, i\in I\}$ and $V=\{y\in A^{\Z\cdot t}: y(-i\cdot t)=a_i, i\in I\}$ for a finite $I\subset \Z$ that is symmetric in the sense of $I=-I$.
    
    We may identify $A^{\Z\cdot t}$ as a copy of $A^\Z$ and define $x_1, x_2\in A^\Z$ by $x_1(n)=x(n)$ and $x_2(n)=x(nt)$. Then any $x\in A^{D_\infty}$ can be written as  $x=(x_1, x_2)$. From this point of view,  $\Z$, as a normal subgroup of $D_\infty$ of index $2$, acts on $A^\Z\times A^{\Z \cdot t }$ diagonally by $n\cdot (x_1, x_2)=(n\cdot x_1, n\cdot x_2)$, where $n\cdot x_i$ is the usual $n$-shift in $A^\Z$ for $i=1, 2$. In addition, in this picture, the element $t$ acts on $x=(x_1, x_2)$ like a mirror by $t\cdot (x_1, x_2)=(\bar{x_2}, \bar{x_1})$. Indeed, let $x=(x_1, x_2)$ observe that \[(t\cdot x)(n)=x(tn)=x(-n\cdot t)=x_2(-n)=\bar{x_2}(n)\]
    and
    \[(t\cdot x)(nt)=x(tnt)=x(-n)=x_1(-n)=\bar{x_1}(n).\]

 Denote by $U'=U\times A^{\Z\cdot t}$ for simplicity. Then, by our construction and Remark \ref{rmk: recurrent set}, observe that 
 \[R_{w, U, \Z}\sqcup t\cdot (-R_{w, U, \Z})\subset R_{\hat{w}, U', D_\infty}.\] 
 Similarly, denote by $V'=A^{\Z}\times V$ and by Remark \ref{rmk: recurrent set}, one has
 \[(-R_{w, U, \Z})\sqcup t\cdot R_{w, U, \Z}\subset R_{\hat{w}, V', D_\infty}.\]
Since $w$ is balanced almost periodic, there is a syndetic set $P\subset R_{w, U, \Z}$ for $\Z$ such that $P=-P$. Therefore, one has 
\[P\sqcup tP\subset (R_{w, U, \Z}\sqcup t\cdot (-R_{w, U, \Z}))\cap ((-R_{w, U, \Z})\sqcup t\cdot R_{w, U, \Z})\subset R_{\hat{w}, U\times V, D_\infty}.\]
Finally, let $K\subset \Z$ be a finite set such that $K+P=\Z$ because $P$ is syndetic in $\Z$. This then implies that $(K\cup tKt)\cdot (P\sqcup tP)=D_\infty$, which implies that $R_{\hat{w}, U\times V, D_\infty}$ is syndetic in $D_\infty$.
\end{proof}

Let $\Gamma\curvearrowright A^\Gamma$ and $w\in A^\Gamma$. Denote by $O(w)=\overline{\Gamma\cdot w}$ the orbit closure of $w$. Now, we have the following.

\begin{thm}
    Let $w$ be a balanced almost periodic word in $A^\Z$ and denote by $\hat{w}\in A^{D_\infty}$ defined above. Then the subshift $\alpha: D_\infty\curvearrowright O(\hat{w})$ is minimal topologically free but not free. Therefore, $\alpha$ is almost finite and $C(O(\hat{w}))\rtimes_r D_\infty$ is $\CZ$-stable and classifiable by the Elliott invariant.
\end{thm}
\begin{proof}
Write $D_\infty=\Z\rtimes \Z_2=\langle s,t| t^2, tsts\rangle$. Lemma \ref{lem: almost periodic} implies that $D_\infty\curvearrowright O(\hat{w})$ is minimal and thus topologically free by \cite[Proposition 2.6]{J}. Note that $\hat{w}$ is a fixed point for $t$ since $t$ acts like a mirror. Thus $\alpha$ is not free.  Then Corollaries \ref{cor: virtually cyclic af} and \ref{cor: final}(ii) apply here.
    \end{proof}

We remark that not all virtually $\Z$ groups admit minimal topologically free but non-free actions. For example, it is well-known that any minimal action of $\Z$ on a compact Hausdorr space has to be free. The following is a generalization of this fact. We remark that Yongle Jiang also independently obtained a proof of the following result. 
\begin{prop}\label{prop: genuine free}
    Let $\alpha: \Z\times F\curvearrowright X$ be a minimal topologically free action on a compact Hausdorff space $X$ in which $F$ is a finite group. Then $\alpha$ has to be free.
\end{prop}
\begin{proof}
    Let $x\in X$ and denote by $\operatorname{stab}(x)$ the stabilizer group at $x$. Suppose $(n, e_F)\in \operatorname{stab}(x)$ for some $n\neq 0$. Then one has $n\Z\times \{e_F\}\leq \operatorname{stab}(x)$, which implies that $[\Gamma: \operatorname{stab}(x)]<\infty$ and thus the orbit of $x$ is finite. This is a contradiction to the assumption that $\alpha$ is minimal. Thus, one has the intersection $\operatorname{stab}(x)\cap (\Z\times \{e_F\})=\{(0, e_F)\}$.
Then suppose $(n,f)\in \operatorname{stab}(x)$ for some $f\neq e_F$. Let $f^k=e_F$ for a $k\in \N_+$. Then one has $x=(n, f)^kx=(kn, e_F)x$, which implies that $(kn, e_F)\in \operatorname{stab}(x)$. This is a contradiction to $\operatorname{stab}(x)\cap (\Z\times \{e_F\})=\{(0, e_F)\}$ obtained above.

Therefore, one has $\operatorname{stab}(x)\leq \{0\}\times F$ for any $x\in X$. For any $f\in F$, define $O_f=\{x\in X: (0, f)x\neq x\}$, which is open dense because $\alpha$ is topologically free. Define $O=\bigcap_{f\in F}O_f$, which is still open dense because $F$ is finite.  Observe that every $x\in O$ is a free point for $\alpha$. Since $\alpha$ is minimal, one has $\Gamma\cdot O=X$. This implies that any $y\in X$ is located on an orbit of a free point $x\in O$. Thus, $y$ is a free point for $\alpha$ itself.
\end{proof}

Nevertheless, from Lemma \ref{lem: almost periodic}, one may still construct more minimal topologically free but non-free subshifts of virtually $\Z$ groups. Let $\Gamma$ be a virtually $\Z$ group. Let $N\triangleleft \Gamma$ be the normal subgroup in $\Gamma$ such that $N\simeq \Z$ and the index $[\Gamma: N]<\infty$. Then the conjugation action of $\Gamma$ on $N$ yields a map $\varphi$ from $\Gamma$ to $\aut(N)\simeq \Z_2$ with the kernel $\ker(\varphi)=C_\Gamma(N)$, i.e., the centralizer of $N$ in $\Gamma$. 
Moreover, note that $N$ is also a normal subgroup of $C_\Gamma(N)$ of finite index. This necessarily implies that $C_\Gamma(N)\simeq N\times F$ for a finite group by looking at the transfer map $\rho: C_\Gamma(N)\to N$, which is surjective. If $\varphi$ is trivial, then $\Gamma=C_\Gamma(N)\simeq \Z\times F$. Then Proposition \ref{prop: genuine free} entails that all minimal topologically free action of such a $\Gamma$ must be free. In the case that $\varphi$ is not trivial, one has the following exact sequence
\[\begin{tikzcd}
0\arrow[r] & \Z\times F \arrow[r, "i"] & \Gamma\arrow[r, "\varphi"] & \Z_2\arrow[r] & 0.
\end{tikzcd}\]
In general, this exact sequence does not split. But in the case that $\Z_2$ acts on $\Z$ in the usual way and acts on $F$ trivially, one obtains a group $\Gamma=(\Z\times F)\rtimes \Z_2$, which is isomorphic to $D_\infty\times F$. In this case, we will construct a minimal topologically free but non-free subshift as follows.

Let $g, h\in D_\infty$ and $s, t\in F$. Then by definition, for any $x\in A^\Gamma$, one has
\[((g, t)\cdot x)(h, s)=x(g^{-1}h, t^{-1}s).\]
View $A^\Gamma=\prod_{s\in F}A^{D_\infty\times \{s\}}$ as usual and write $x\in A^\Gamma$ by $x=(x_s: s\in F)$, where $x(s)=x_s\in A^{D_\infty\times\{s\}}$ as a copy of $A^{D_\infty}$. Then, $D_\infty$ acts on $A^\Gamma$ diagonally by \[g\cdot (x_s: s\in F)=(g\cdot x_s: s\in F),\] 
where $g\cdot x_s$ is the shift for $x_s$ by $g$ in $A^{D_\infty}$.
Therefore, from this point of view, the shift action of $\Gamma$ on $A^\Gamma$ is of the form
\[((g, t)\cdot x)_s=g\cdot x_{t^{-1}s}.\]

Now choose an alphabet $A=\{0, a_s: s\in F\}$ such that any two symbols in $A$ are different, i.e., $|A|=|F|+1$. Now choose an almost periodic  $w\in \{0, 1\}^{D_\infty}$ as in Lemma \ref{lem: almost periodic}. Then for each $s\in F$, we define a word $w_s\in \{0, a_s\}^{D_\infty}$ by replacing all ``$1$'' in $w$  by ``$a_s$''. Note that all of these $w_s$ are in $A^{D_\infty}$. Now define $x_0\in A^\Gamma$ by setting $x_0(s)=w_s$ for $s\in F$.

\begin{lem}\label{lem: AP for virtually Z}
    The word $x_0\in A^\Gamma$ defined above is almost periodic.
\end{lem}
\begin{proof}
    Identify $A^\Gamma= \prod_{s\in F}A^{D_\infty\times \{s\}}\simeq\prod_{s\in F}A^{D_\infty}$ and let $x_0=(w_s: s\in F)$ be defined above and $O$ an open neighborhood of $x_0$. Without loss of any generality, one may assume $O=\prod_{s\in F} U_s$ and each $U_s$ is cylinder set in $A^{D_\infty}$ with the form $U_s=\{y\in A^\Gamma: y(i)=w_s(i), i\in I_s\}$ for some finite $I_s\subset D_\infty$. Shrink all $U_s$ if necessary, we may assume there is an $I\subset D_\infty$ and $I=I_s$ for any $s\in F$. By our construction of each $w_s$ from $w$, the recurrent set $R_{w_s, U_s, D_\infty}$ are the same set in $D_\infty$ for any $s\in F$, denoted by $R$ for simplicity. Now because $w$ has been chosen as an almost periodic word for $D_\infty$-shift, the set $R$ is syndetic in $D_\infty$ and thus syndetic in $\Gamma=D_\infty\times F$. This implies that $R_{x_0, O, \Gamma}$ is syndetic in $\Gamma$ because $R\subset R_{x_0, O, \Gamma}$ and thus $x_0$ is almost periodic.
\end{proof}

\begin{thm}
    Let $\Gamma=D_\infty\times F$ and $x_0$ be defined above. Then the subshift $\alpha: \Gamma\curvearrowright O(x_0)$ is minimal topologically free but not free. Therefore, $\alpha$ is almost finite and
    $C(O(x_0))\rtimes_r \Gamma$ is $\CZ$-stable and classifiable by the Elliott invariant.
      \end{thm}
      \begin{proof}
          It follows from Lemma \ref{lem: AP for virtually Z} that the subshift $\alpha: \Gamma\curvearrowright O(x_0)$ is minimal. To show $\alpha$ is topologically free, it suffices to find one free word $y_0$ in $O(x_0)$. This is because the set of points with trivial stabilizers is a $G_\delta$ set and all points in the dense orbit $\Gamma\cdot y_0$ are free points. 
          
          Recall $w\in \{0, 1\}^{D_\infty}$ is almost periodic. Then $D_\infty\curvearrowright \overline{D_\infty\cdot w}$ is minimal and thus topologically free by \cite[Proposition 2,6]{J}. Choose a $D_\infty$-free word $y\in \overline{D_\infty\cdot w}$ and define $y_s$ to be in $\{0, a_s\}^{D_\infty}$ by replacing all ``$1$'' in $y$ by ``$a_s$''.   
          Now define a word $y_0=(y_s: s\in F)\in A^\Gamma$. Note that $y_0\in O(x_0)$ by construction.
           
           We claim that $y_0$ is a $\Gamma$-free word. Indeed, let $g\in D_\infty $ and $t\in F$ and suppose $(g,t)\cdot y_0=y_0$, which means $g\cdot y_{t^{-1}s}=y_s$. However, recall $g\cdot y_{t^{-1}s}\in \{0, a_{t^{-1}s}\}^{D_\infty}$ and $y_s\in \{0, a_s\}^{D_\infty}$. This implies that $t=e_F$. Then, $g\cdot y_s=y_s$ implies that $g=e_{D_\infty}$ since $y_s$ is a $D_\infty$-free word. Therefore, $y_0$ is a $\Gamma$-free word and thus the action $\alpha$ is topologically free.

           Finally, write $D_\infty=\Z\rtimes \Z_2$, for which we denote by $g_0$ the generator for $\Z_2$.
           Recall the $w$ constructed from Lemma \ref{lem: almost periodic} is a fixed point for $g_0$. This implies that $w_s$ are fixed points for $g_0$ and thus $x_0=(w_s: s\in F)$ is a fixed point for $g_0$. Thus, $\alpha$ is not free.

           Now, we apply Corollaries \ref{cor: virtually cyclic af} and \ref{cor: final}(ii) here.
      \end{proof}

\section{Acknowlegement}
The authors are grateful to the Fields Institute for hosting authors
during the Thematic Program on Operator Algebras in Fall 2023 (K. Li as a visitor and X. Ma as a postdoc). 
The authors would also like to thank Yongle Jiang, N. Christopher Phillips, Eduardo Scarparo, and Jianchao Wu for the helpful discussion and useful comments. Finally, we would like to thank the anonymous reviewers for the helpful comments and for pointing out an obvious inaccuracy.

\end{document}